\documentclass[a4paper,11pt]{amsart}

\usepackage[top=30truemm,bottom=25truemm,left=35truemm,right=35truemm]{geometry}

\usepackage{amsmath}
\usepackage{amssymb}
\usepackage{amsthm}
\usepackage{graphics}
\usepackage[abbrev,alphabetic]{amsrefs}
\usepackage{amscd}
\usepackage[dvipdfmx]{graphicx}
\usepackage{url}
\usepackage{ulem}

\newcommand{\doi}[1]{\url{https://doi.org/#1}}

\usepackage[all,cmtip]{xy}
\usepackage{xypic}

\title[A counterexample to the PIA conjecture]
{A counterexample to the PIA conjecture for minimal log discrepancies}

\author{Yusuke Nakamura}
\address{Graduate School of Mathematics, Nagoya University, Furo-cho, Chikusa-ku, Nagoya, 464-8602, Japan.}
\email{y.nakamura@math.nagoya-u.ac.jp}
\urladdr{https://sites.google.com/site/ynakamuraagmath/}

\author{Kohsuke Shibata}
\address{School of Engineering, 
Tokyo Denki University, Adachi-ku, Tokyo 120-8551, Japan.}
\email{shibata.kohsuke@mail.dendai.ac.jp}


\keywords{minimal log discrepancy, precise inversion of adjunction, PIA conjecture, LSC conjecture, hyperquotient singularities}

\newtheorem{thm}{Theorem}[section]
\newtheorem{lem}[thm]{Lemma}

\newtheorem{prop}[thm]{Proposition}

\theoremstyle{definition}
\newtheorem{defi}[thm]{Definition}
\newtheorem{eg}[thm]{Example}
\newtheorem{conj}[thm]{Conjecture}

\theoremstyle{remark}
\newtheorem{rmk}[thm]{Remark}
\newtheorem{quest}[thm]{Question}
\newtheorem*{ackn}{Acknowledgements}

\begin{document}
\begin{abstract}
We give a counterexample to the PIA (precise inversion of adjunction) conjecture for minimal log discrepancies. 
We also give a counterexample to the LSC conjecture for families. 
\end{abstract}

\maketitle

\section{Introduction}
The minimal log discrepancy is an invariant of singularities in birational geometry. 
The PIA (precise inversion of adjunction) conjecture is a fundamental conjecture on the minimal log discrepancy, 
as well as the ACC (ascending chain condition) conjecture and the LSC (lower semi-continuity) conjecture. 
The ACC conjecture and the LSC conjecture are important, 
as Shokurov showed that the termination of flips conjecture follows from these conjectures (\cite{Sho04}). 
The PIA conjecture is also considered important like the usual inversion of adjunction when applying induction arguments on dimension.
It also has applications, such as reducing the LSC conjecture to the case where the ambient space has milder singularities (\cite{EM04}*{Theorem 1.2}). 

In this paper, we always work over an algebraically closed field $k$ of characteristic zero. 
The usual inversion of adjunction is the following theorem.
\begin{thm}[Inversion of adjunction, {\cite{92}*{Theorem 17.6}, \cite{Kaw07}}]\label{thm:IA}
Let $(X, \Delta)$ be a log pair over $k$ and let $D$ be a normal Cartier divisor such that $D \not \subset \operatorname{Supp} (\lfloor \Delta \rfloor )$. 
\begin{enumerate}
\item $(X, D + \Delta)$ is plt around $D$ if and only if $(D, \Delta |_D)$ is klt. 
\item $(X, D + \Delta)$ is log canonical around $D$ if and only if $(D, \Delta |_D)$ is log canonical.
\end{enumerate}
\end{thm}
\noindent
Theorem \ref{thm:IA}(1) was proved by Koll{\'a}r (\cite{92}*{17.6}), and Theorem \ref{thm:IA}(2) by Kawakita \cite{Kaw07}. 
See \cite{92} and \cite{Kaw07} for more general statements when $D$ is not normal or not Cartier. 
The PIA (precise inversion of adjunction) conjecture asserts that such equivalences hold at the level of minimal log discrepancies: 
\begin{conj}[PIA conjecture, {\cite{92}*{17.3.1}}]\label{conj:PIA2}
Let $(X, \Delta)$ be a log pair over $k$ and let $D$ be a normal Cartier prime divisor on $X$ such that $D \not \subset \operatorname{Supp} (\lfloor \Delta \rfloor )$. 
Let $x \in D$ be a closed point. 
Then 
\[
\operatorname{mld}_x ( X, D + \Delta \bigr) = \operatorname{mld}_x (D, \Delta |_D)
\]
holds. 
\end{conj}
\noindent
In this paper, we consider the following formulation using $\mathbb{R}$-ideals.
Note that any $\mathbb{Q}$-Cartier divisor can be associated with an $\mathbb{R}$-ideal (see Subsection \ref{subsection:LP}).  Therefore, if $X$ is $\mathbb{Q}$-Gorenstein, Conjecture \ref{conj:PIA2} is a special case of Conjecture \ref{conj:PIA}. 
\begin{conj}[PIA conjecture, {\cite{EMY03}*{Conjecture 1.5}}]\label{conj:PIA}
Let $(X, \mathfrak{a})$ be a log pair over $k$ and let $D$ be a normal Cartier prime divisor on $X$. 
Let $x \in D$ be a closed point. 
Suppose that $D$ is not contained in the cosupport of the $\mathbb{R}$-ideal sheaf $\mathfrak{a}$. 
Then 
\[
\operatorname{mld}_x \bigl( X, \mathfrak{a} \mathcal{O}_X(-D) \bigr) = \operatorname{mld}_x (D, \mathfrak{a} \mathcal{O}_D)
\]
holds. 
\end{conj}
\noindent
The inequality ``$\le$'' is known to always hold (\cite{92}*{Theorem 17.2}).  
Therefore, the opposite inequality ``$\ge$'' constitutes the essential part of the conjecture.

Conjecture \ref{conj:PIA} is known to be true in the following cases: 
\begin{itemize}
\item[({\ref{conj:PIA}}.1)]
the case when $\dim X = 2$ by Shokurov (\cite{Sho93}). 

\item[({\ref{conj:PIA}}.2)]
the case when $X$ is smooth by Ein, Musta{\c{t}}{\v{a}}, and Yasuda (\cite{EMY03}). 

\item[({\ref{conj:PIA}}.3)]
the case when $X$ is a locally complete intersection variety by Ein and Musta{\c{t}}{\v{a}} (\cite{EM04}). 

\item[({\ref{conj:PIA}}.4)]
the case when $\dim X = 3$ (see Proposition \ref{prop:dim3}). 

\item[({\ref{conj:PIA}}.5)]
the case when $X$ has only quotient singularities and $D$ is klt at $x$ by the authors (\cites{NS22, NS2}). 

\item[({\ref{conj:PIA}}.6)]
In \cites{NS22, NS2} (see also \cite{NS3}), the authors more generally prove the PIA conjecture for the case when $X$ is a subvariety of a variety $Y$ such that 
\begin{itemize}
\item 
$Y$ has only quotient singularities at $x$, 

\item 
$X$ is locally defined by $c$ equations in $Y$ at $x$ for $c := \dim Y - \dim X$, and 

\item 
$X$ and $D$ are klt at $x$. 
\end{itemize}
\end{itemize}
See also \cite{BCHM}*{Corollary 1.4.5} for a related result, where a version of the precise inversion of adjunction is proved.
We emphasize that the klt assumption on the singularities of $D$, as in ({\ref{conj:PIA}}.5) and ({\ref{conj:PIA}}.6), is not required in the original conjecture or the known results, namely Theorem \ref{thm:IA} and ({\ref{conj:PIA}}.1)--({\ref{conj:PIA}}.4). 

The purpose of this paper is to give a counterexample to the PIA conjecture. 

\begin{eg}\label{eg:main}
Let $\xi \in k$ be a primitive third root of unity. 
Let $n \in \mathbb{Z}_{\ge 0}$.
We define $R := k[x_1, x_2, \ldots , x_{n+3}]$ and 
\[
f := x_1 ^3 + x_2 ^3 + x_3 ^3 \in R, \qquad
\gamma := \operatorname{diag} \bigl( 1, \xi, \xi^2, \overbrace{\xi , \cdots , \xi}^{n} \bigr) \in \operatorname{GL}_{n+3}(k). 
\] 
We define $G := \langle \gamma \rangle = \{ 1, \gamma , \gamma^2 \} \subset \operatorname{GL}_{n+3}(k)$, and 
\begin{align*}
\overline{A}_n &:= \operatorname{Spec}(R) = \mathbb{A}^{n+3}_k, & A_n&:= \overline{A}_n/G = \operatorname{Spec}\bigl( R^G \bigr), \\
\overline{B}_n &:= \operatorname{Spec}(R/(f)), & B_n&:= \overline{B}_n/G = \operatorname{Spec} \bigl( R^G/(f) \bigr). 
\end{align*}
Let $\overline{x}_0 = (0, \cdots, 0)$ denote the origin of $\overline{A}_n$, 
and let $x_0 \in A_n$ denote the image of $\overline{x}_0$. 
\end{eg}

\begin{thm}\label{thm:main}
In Example \ref{eg:main}, we have the following assertions. 
\begin{itemize}
\item[(1)] $\operatorname{mld}_{x_0}(A_n, B_n)  \le 1 + \frac{n}{3}$ (see Appendix \ref{section:exact} for the exact value). 
\item[(2)] $\operatorname{mld}_{x_0}(B_n)  = n$. 
\item[(3)] When $n \ge 2$, 
Conjecture \ref{conj:PIA} does not hold for $(X, D) = (A_n, B_n)$ and $\mathfrak{a} = \mathcal{O}_{A_n}$. 

\item[(3')] When $n \ge 2$, Conjecture \ref{conj:PIA2} does not hold for $(X, D) = (A_n, B_n)$ and $\Delta = 0$.
\end{itemize}
\end{thm}
Thus, the PIA conjecture has a counterexample when $\dim X \ge 5$. 
The conjecture is still open when $\dim X = 4$ (see  Question \ref{quest:dim4}). 
Note that $B_n$ is not klt at $x_0$. 
Therefore, it shows that the assumption ``$D$ is klt at $x$'' in the authors' result (\ref{conj:PIA}.5) above was necessary.  
It also suggests that the PIA conjecture might be modified to add the additional assumption ``$D$ is klt at $x$'' (see Question \ref{quest:PIAm}). 

The same example can be used to construct a counterexample to the LSC conjecture for families (see Example \ref{eg:family}).
The LSC conjecture for families (Conjecture \ref{conj:LSCg}) can naturally be viewed as an mld-version of the deformation invariance of terminal singularities proved by Nakayama (\cite{Nakbook}*{Ch.VI\ Corollary 5.3(2)}) and canonical singularities proved by Kawamata (\cite{Kaw99}).  
We refer the reader to \cite{IshiiBook}*{Section 9} for further details on the deformation invariance of other types of singularities and the semi-continuity of other invariants.  
Example \ref{eg:family} shows that Conjecture \ref{conj:LSCg} fails when the fibers are not klt. 
Note that the LSC conjecture for families is a stronger statement than the original LSC conjecture (Conjecture \ref{conj:LSC}), 
and Example \ref{eg:family} does not disprove the original LSC conjecture.

An interesting point of Example \ref{eg:main} is that we have
\[
\operatorname{mld}_{x_0}(B_n)  = \operatorname{mld}_{\overline{x}_0}(\overline{B}_n) \tag{$\spadesuit$}
\]
even though the $G$-action on $\overline{B}_n$ is not free. 
To explain this phenomenon ($\spadesuit$), we introduce a new concept called ``virtually free'' for a finite group action (see Section \ref{section:ef}). 
This concept lies between ``free'' and ``free in codimension one''. 
We prove that the minimal log discrepancy remains unchanged after taking the quotient of a virtually free action (Proposition \ref{prop:eq}). 
In order to prove the $G$-action on $\overline{B}_n$ is virtually free, 
we prove a permanence property of the virtual freeness when taking the product of varieties (Theorem \ref{thm:prod}). 

We prove that the virtual freeness is equivalent to the freeness when the considered variety is klt (Theorem \ref{thm:kltcase}). 
This fact explains why it is not possible to make a similar counterexample with only klt singularities. 
We note that the phenomenon $(\spadesuit)$ can also be explained from the perspective of arc space theory (see Remark \ref{rmk:thin}).

The paper is organized as follows. 
In Section \ref{section:ef}, we introduce the new concept called ``virtual freeness'' and discuss its properties. 
In Section \ref{section:proof}, we give a proof of Theorem \ref{thm:main}. 
In Section \ref{section:LSC}, we discuss two pathological examples related to the semi-continuity of minimal log discrepancies (Examples \ref{eg:family}, \ref{eg:adic}). 
In Section \ref{section:Q}, we give a proof of the PIA conjecture in dimension three, which seems to be well-known to experts (Proposition \ref{prop:dim3}). 
We also give some related questions (Questions \ref{quest:dim4}, \ref{quest:PIAm}). 
In Appendix \ref{section:exact}, we determine the exact value of $\operatorname{mld}_{x_0}(A_n, B_n)$.

\begin{ackn}
We would like to thank Professors Shigeharu Takayama, Takayuki Koike, and Chen Jiang for discussions. 
We would also like to thank the referees for their valuable comments and suggestions, 
which have improved the quality of the paper.
The first author is partially supported by Inamori Foundation and by JSPS KAKENHI No.\ 18K13384 and 22K13888. 
The second author is partially supported by JSPS KAKENHI No.\ 19K14496 and 23K12958.
\end{ackn}

\section{Preliminaries}\label{section:prelimi}

\subsection{Notation}\label{subsection:notation}

\begin{itemize}
\item 
We basically follow the notations and the terminologies in \cite{Har77} and \cite{Kol13}.

\item 
Throughout this paper, $k$ is an algebraically closed field of characteristic zero. 
We say that $X$ is a \textit{variety over} $k$ if 
$X$ is an integral scheme that is separated and of finite type over $k$. 
\end{itemize}

\subsection{Log pairs}\label{subsection:LP}
A \textit{log pair} $(X, \mathfrak{a})$ is a normal $\mathbb{Q}$-Gorenstein variety $X$ over $k$ and 
an $\mathbb{R}$-ideal sheaf $\mathfrak{a}$ on $X$. 
Here, an $\mathbb{R}$-\textit{ideal sheaf} $\mathfrak{a}$ on $X$ is a formal product 
$\mathfrak{a} = \prod _{i = 1} ^s \mathfrak{a}_i ^{r_i}$, where $\mathfrak{a}_1, \ldots, \mathfrak{a}_s$ are non-zero coherent ideal sheaves on $X$ and $r_1, \ldots , r_s$ are positive real numbers. 
For a morphism $Y \to X$ and an $\mathbb{R}$-ideal sheaf $\mathfrak{a} = \prod _{i = 1} ^s \mathfrak{a}_i ^{r_i}$, 
let $\mathfrak{a} \mathcal{O}_Y$ denote the $\mathbb{R}$-ideal sheaf $\prod _{i = 1} ^s (\mathfrak{a}_i \mathcal{O}_Y)  ^{r_i}$ on $Y$. 

Let $(X, \mathfrak{a} = \prod _{i = 1} ^s \mathfrak{a}_i ^{r_i})$ be a log pair. 
Let $f: X' \to X$ be a proper birational morphism from a normal variety $X'$, and let $E$ be a prime divisor on $X'$.  
Then, the \textit{log discrepancy} of $(X, \mathfrak{a})$ at $E$ is defined as 
\[
a_E(X, \mathfrak{a}) := 1 + \operatorname{coeff}_E (K_{X'} - f^* K_X) - \operatorname{ord}_E ( \mathfrak{a} ), 
\]
where $\operatorname{ord}_E ( \mathfrak{a} ) := \sum _{i=1} ^s r_i \operatorname{ord}_E ( \mathfrak{a}_i )$. 
The image $f(E)$ is called the \textit{center} of $E$ on $X$ and denoted by $c_X(E)$. 
For a closed subset $W \subset X$, the \textit{minimal log discrepancy} along $W$ is defined as 
\[
\operatorname{mld}_W (X, \mathfrak{a}) := \inf _{c_X(E) \subset W} a_E (X, \mathfrak{a})
\]
if $\dim X \ge 2$, where the infimum is taken over all prime divisors $E$ over $X$ with center $c_X(E) \subset W$. 
When $\dim X = 1$, we define $\operatorname{mld}_W (X, \mathfrak{a}) := \inf _{c_X(E) \subset W} a_E (X, \mathfrak{a})$ 
if the infimum is non-negative and 
$\operatorname{mld}_W (X, \mathfrak{a}) := - \infty$ otherwise. 
It is known that $\operatorname{mld}_W (X, \mathfrak{a}) \in \mathbb{R}_{\ge 0} \cup \{ - \infty \}$ (cf.\ \cite{KM98}*{Corollary 2.31}). 

Let $D$ be a $\mathbb{Q}$-Cartier divisor on $X$. 
Take $\ell \in \mathbb{Z} _{>0}$ such that $\ell D$ is Cartier. 
Then, we define $a_E (X, D) := a_E \bigl( X, (\mathcal{O}_X(- \ell D)) ^{\frac{1}{\ell}} \bigr)$ and $\operatorname{mld}_W (X, D) := \operatorname{mld}_W \bigl( X, (\mathcal{O}_X(- \ell D)) ^{\frac{1}{\ell}} \bigr)$. 
Note that these values do not depend on the choice of $\ell$.

When $\mathfrak{a} = \mathcal{O}_X$, we define $a_E (X) := a_E (X, \mathfrak{a})$ and $\operatorname{mld}_W (X) := \operatorname{mld}_W (X, \mathfrak{a})$ for simplicity. 
When $W = \{ x \}$ for some closed point $x \in X$, 
we denote $\operatorname{mld}_{\{ x \}}$ by $\operatorname{mld}_x$ for simplicity. 

The minimal log discrepancy of a quotient singularity is known to be computable via the age function. 
\begin{defi}[\cite{NSs}*{Definition 2.3}]\label{defi:age}
Let $n$ be a positive integer and 
let $G \leqslant \operatorname{GL}_{n}(k)$ be a finite subgroup. 
Let $d := \# G$ be the order of $G$, and let $\xi \in k$ be a primitive $d$th root of unity. 
Since each $g \in G$ has finite order, $g$ is conjugate to a diagonal matrix $\operatorname{diag}(\xi ^{e_1}, \ldots , \xi ^{e_n})$ 
with $1 \le e_i \le d$. 
Then, we define $\operatorname{age}'(g) := \sum _{i = 1} ^n \frac{e_i}{d}$. 
\end{defi}
\begin{prop}[\cite{NSs}*{Proposition 2.8}]\label{prop:age}
Let $n$ be a positive integer and 
let $G \leqslant \operatorname{GL}_{n}(k)$ be a finite subgroup. 
Suppose that $G$ does not contain a pseudo-reflection. 
Let $x_0 \in \mathbb{A}^n_k / G$ be the image of the origin of $\mathbb{A}^n _k$. 
Then, we have
\[
\operatorname{mld}_{x_0} \bigl( \mathbb{A}^n_k / G \bigr) = \min \{ \operatorname{age}'(g) \mid g \in G \}.
\]
\end{prop}
\noindent
In this paper, we apply Proposition \ref{prop:age} only in the case of cyclic groups.

\section{Virtually free action}\label{section:ef}
In this paper, an action of a finite group $G$ on a variety $X$ is called \textit{free in codimension one} 
if it is free on the set of closed points, $|X|_{\rm cl}$, away from a closed subset of codimension at least two.

In this section, we introduce a new concept called ``virtually free'' for a finite group action. 
This concept lies between ``free'' and ``free in codimension one''. 
In Proposition \ref{prop:eq}, we prove that the minimal log discrepancy remains unchanged after taking the quotient of a virtually free action. 
In Proposition \ref{prop:char}, we provide several characterizations of the virtual freeness. 
In Theorem \ref{thm:prod}, we prove a permanence property of the virtual freeness when taking the product of varieties. 
In Theorem \ref{thm:kltcase}, we show that the virtual freeness is equivalent to the freeness when the considered variety is klt. 

\begin{defi}\label{defi:ess}
Let $G$ be a finite group, and let $X$ be a normal variety. 
We say that a $G$-action on $X$ is \textit{virtually free} if the following condition holds: 
\begin{itemize}
\item 
For any $G$-equivariant proper birational morphism $Y \to X$ from a normal variety $Y$ with a $G$-action, 
the $G$-action on $Y$ is free in codimension one. 
\end{itemize}
\end{defi}

\begin{rmk}\label{rmk:between}
By the definition, we have implications $(1) \Rightarrow (2) \Rightarrow (3)$ among the following three conditions:  
\begin{enumerate}
\item a $G$-action on $X$ is free. 
\item a $G$-action on $X$ is virtually free. 
\item a $G$-action on $X$ is free in codimension one.
\end{enumerate}
In Theorem \ref{thm:kltcase}, we will see $(2) \Rightarrow (1)$ when $X$ is klt. 
In Lemma \ref{lem:H}, we give an example that is virtually free but not free. 
\end{rmk}

\begin{lem}\label{lem:equiv_resol}
Suppose that a finite group $G$ acts on a normal variety $X$. 
Let $E$ be a divisor over $X$. 
Then, there exists a $G$-equivariant proper birational morphism $Y \to X$ from a normal variety $Y$ with a $G$-action such that $E$ is a prime divisor on $Y$. 
\end{lem}
\begin{proof}
Lemma 2.45 in [KM98] says that, by repeatedly blowing up along the center of $E$, we obtain a sequence
$Z_n \to Z_{n-1} \to \cdots \to Z_0=X$ such that $E$ becomes a divisor on $Z_n$. 
Let $W_n \to W_{n-1} \to \cdots \to W_0=X$ be the sequence of blow-ups along the $G$-orbit of the center of $E$ (i.e., $W_{i+1} \to W_i$ is the blow-up of $W_i$ along $\bigcup _{g \in G} g \big( c_{W_i}(E) \big)$). 
Then, all blow-ups in this sequence are $G$-equivariant. 
Furthermore, $E$ becomes a divisor on $W_n$ since each blow-up $W_{i+1} \to W_{i}$ is isomorphic to 
the corresponding blow-up $Z_{i+1} \to Z_i$ at the generic points of the centers of $E$ on $W_i$ and $Z_i$. 
Therefore, the normalization of $W_n$ is the desired variety. 
\end{proof}

\begin{lem}\label{lem:normalization}
Suppose that a finite group $G$ acts on a normal variety $X$. 
Let $W \to X/G$ be a proper birational morphism from a normal variety $W$. 
Let $W'$ be the normalization of $W$ in the field $k(X)$ of fractions of $X$. 
Then, the $G$-action on $X$ extends to $W'$, and we have $W' / G \simeq W$. 
\end{lem}
\begin{proof}
Let $U = \operatorname{Spec} R$ be an affine open subset of $W$, and let $S$ be the normalization of $R$ in $k(X)$. 
It is sufficient to show that $S ^G = R$. 
We have $k(X) ^G = k(X/G) = Q(R)$, where $Q(R)$ denotes the field of fractions of $R$. 
Therefore, we have 
\[
S ^G = S \cap k(X)^G = S \cap Q(R) = R. 
\]
Here, the last equality follows from the normality of $R$. 
\end{proof}

\begin{prop}\label{prop:eq}
Suppose that a finite group $G$ acts on a normal $\mathbb{Q}$-Gorenstein variety $X$. 
Suppose that the $G$-action on $X$ is virtually free. 
Let $\mathfrak{a}$ be a non-zero $\mathbb{R}$-ideal sheaf on $X/G$. 
Then, we have 
\[
\operatorname{mld}_{x}(X, \mathfrak{a}\mathcal{O}_X) = \operatorname{mld}_{x'}(X/G, \mathfrak{a})
\]
for any closed point $x \in X$ and its image $x'$ on the quotient variety $X/G$. 
\end{prop}
\begin{proof}
Let $Y \to X$ be any $G$-equivariant proper birational morphism from a normal variety $Y$. 
Let $E$ be a prime divisor on $Y$, and $F$ the image of $E$ on $Y/G$. 
As stated in Remark \ref{rmk:between}, $G$ acts freely on both $X$ and $Y$ in codimension one. 
Since $X \to X/G$ and $Y \to Y/G$ are \'{e}tale in codimension one, it follows that $a_E(X, \mathfrak{a}\mathcal{O}_X) = a_F(X/G, \mathfrak{a})$ (cf.\ \cite{Kol13}*{2.42.4}). 

Let $E$ be a divisor over $X$. 
By Lemma \ref{lem:equiv_resol}, there exists a $G$-equivariant proper birational morphism $Z \to X$ from a normal variety $Z$ such that $E$ is a prime divisor on $Z$. 
Let $F$ be the image of $E$ on $Z/G$. 
Then, by the conclusion from the first paragraph of this proof, we have 
\[
a_E (X, \mathfrak{a}\mathcal{O}_X) = a_F(X/G, \mathfrak{a}) \ge \operatorname{mld}_{x'}(X/G, \mathfrak{a}). 
\]
Therefore, we have $\operatorname{mld}_{x}(X, \mathfrak{a}\mathcal{O}_X) \ge \operatorname{mld}_{x'}(X/G, \mathfrak{a})$.

Conversely, let $F$ be a divisor over $X/G$, 
and suppose that $F$ appears as a divisor on some normal variety $W$. 
By Lemma \ref{lem:normalization}, there exists a $G$-equivariant proper birational morphism $W' \to X$ from a normal variety $W'$ such that $W'/G \simeq W$.
Take a divisor $E$ on $W'$ whose image on $W$ is $F$. 
Then, again by the conclusion from the first paragraph of this proof, we have 
\[
a_F(X/G, \mathfrak{a}) = a_E(X, \mathfrak{a}\mathcal{O}_X) \ge \operatorname{mld}_{x}(X, \mathfrak{a}\mathcal{O}_X). 
\]
Therefore, we have $\operatorname{mld}_{x'}(X/G, \mathfrak{a}) \ge \operatorname{mld}_{x}(X, \mathfrak{a}\mathcal{O}_X)$.
\end{proof}

\begin{lem}\label{lem:model_change}
Let $G$ be a group. Let $f_1: Y_1 \to X$ and $f_2: Y_2 \to X$ be two $G$-equivariant proper birational morphisms between normal varieties $X$, $Y_1$ and $Y_2$ with $G$-actions. 
Suppose that prime divisors $E_1$ on $Y_1$ and $E_2$ on $Y_2$ are the strict transforms of each other (by the maps $f_2 ^{-1} \circ f_1$ and $f_1 ^{-1} \circ f_2$). 
Then, for $\gamma \in G$, $E_1$ is point-wise fixed by $\gamma$ if and only if $E_2$ is point-wise fixed by $\gamma$. 
\end{lem}
\begin{proof}
We can take a normal variety $W$ with a $G$-action equipped with $G$-equivariant proper birational morphisms $W \to Y_1$ and $W \to Y_2$ that are compatible with $f_1$ and $f_2$. 
Therefore, we may assume that $f_2 ^{-1} \circ f_1 : Y_1 \dasharrow Y_2$ is a morphism from the beginning. 
Since $Y_2$ is normal, the morphism $f_2 ^{-1} \circ f_1$ is an isomorphism at the generic point of $E_2$. 
As the set of $\gamma$-fixed points is closed, if an open subset of the divisor $E_i$ (for $i = 1, 2$) is point-wise fixed by $\gamma$, then all points of $E_i$ are point-wise fixed by $\gamma$.
This completes the proof. 
\end{proof}

\begin{lem}\label{lem:wb}
Let $W$ be a variety of dimension $N$ and let $x \in W$ be a closed point. 
Suppose that $W$ is smooth at $x$. 
Let $\gamma : W \to W$ be an automorphism of $W$ with finite order $d$. 
Suppose that $\gamma$ fixes $x$. 
Let $\xi \in k$ be a primitive $d$th root of unity. 
\begin{enumerate}
\item
Then, there exist a regular system of parameters $f_1, \ldots , f_{N} \in \mathcal{O}_{W, x}$ and $a_1, \ldots , a_N \in \{ 1, 2 , \ldots , d \}$ such that 
$\gamma ^* (f_i) = \xi ^{a_i}  f_i$ for each $1 \le i \le N$. 

\item
Let $W'$ be the weighted blow-up of $W$ at $x$ with respect to the regular system of parameters $f_1, \ldots , f_{N}$ and the weight $(a_1, \ldots, a_N)$ in (1). 
That is, $W'=\operatorname{Proj}_W \bigoplus_{\ell \in \mathbb{Z}_{\ge 0}} \mathcal{I}_{\ell} $, where $\mathcal{I}_{\ell}$ is the ideal on $W$ generated by all monomials $f_1^{c_1}\cdots f_{N}^{c_N}$ satisfying $\sum_{i=1}^N a_i c_i \ge \ell$.
Then, the automorphism $\gamma$ on $W$ extends to $W'$ and the exceptional divisor of $W' \to W$ is point-wise fixed by $\gamma$. 
\end{enumerate}
\end{lem}
\begin{proof}
According to the proof of \cite{Kaw24}*{Theorem 2.2.4} (see also \cite{Kaw24}*{Remark 2.2.5}), we can choose a regular system of parameters in $\mathcal{O}_{W, x}$ with respect to which the action of $\gamma$ is linear. 
Then, $\gamma$ corresponds to an element of $\operatorname{GL}_N(k)$ with finite order. 
Since any such element is diagonalizable, (1) holds. 

Next, we prove (2). 
Since $\gamma$ acts on each $\mathcal{I}_{\ell}$, 
it induces an action on $W'=\operatorname{Proj}_W \bigoplus_{\ell \in \mathbb{Z}_{\ge 0}} \mathcal{I}_{\ell}$. 
Next, we show that the exceptional divisor of $W' \to W$ is point-wise fixed by $\gamma$. 
For $\ell \in \mathbb Z_{\ge 0}$, each $\mathcal{I}_{\ell}/\mathcal{I}_{\ell+1}$ is generated over $k$ by the monomials of the form $h = f_1^{b_1} \cdots f_N^{b_N}$ satisfying $a_1 b_1 + \cdots + a_N b_N = \ell$.
Hence, for any $\ell \in \mathbb{Z}_{\ge 0}$ and any such $h \in \mathcal{I}_{d\ell}/\mathcal{I}_{d\ell+1}$, we have $\gamma^*(h) = h$.
Therefore, $\gamma$ acts trivially on $\bigoplus_{\ell \in \mathbb{Z}_{\ge 0}} \mathcal{I}_{d\ell}/\mathcal{I}_{d\ell+1}$.
Since the exceptional divisor of $W' \to W$ is given by
$\operatorname{Proj}_W \bigoplus_{\ell \in \mathbb{Z}_{\ge 0}} \left(\mathcal{I}_{\ell}/\mathcal{I}_{\ell+1}\right) \simeq \operatorname{Proj}_W \bigoplus_{\ell \in \mathbb{Z}_{\ge 0}} \left(\mathcal{I}_{d\ell}/\mathcal{I}_{d\ell+1}\right)$, 
it is point-wise fixed by $\gamma$.
\end{proof}

\begin{prop}\label{prop:char}
Suppose that a finite group $G$ acts on a normal variety $X$. 
Then, the following conditions are equivalent. 
\begin{enumerate}
\item 
The $G$-action on $X$ is virtually free. 

\item
There exists a $G$-equivariant proper birational morphism $Y \to X$ from a normal variety $Y$ such that the $G$-action on $Y$ is free.  

\item
For any $G$-equivariant proper birational morphism $Y \to X$ from a smooth variety $Y$, the $G$-action on $Y$ is free. 
\end{enumerate}
\end{prop}
\begin{proof}
Since we can always take a $G$-equivariant resolution of $X$, the implication $(3) \Rightarrow (2)$ is clear. 

We prove $(2) \Rightarrow (1)$. 
Suppose that (1) does not hold. 
Then, there exist a $G$-equivariant proper birational morphism $Y \to X$ from a normal variety $Y$ 
and a prime divisor $E$ on $Y$ such that $E$ is point-wise fixed by some $\gamma \in G \setminus \{ 1_G \}$. 
Let $Y' \to X$ be any $G$-equivariant proper birational morphism from a normal variety $Y'$. 
We shall prove that the $G$-action on $Y'$ is not free.  
We take a normal variety $W$ with a $G$-action equipped with $G$-equivariant proper birational morphisms $W \to Y$ and $W \to Y'$ that commute with $Y \to X$ and $Y' \to X$. 
Since $W \to Y$ is proper and $Y$ is normal, $W \to Y$ is isomorphic outside a codimension two closed subset of $Y$. 
Therefore, the strict transform $E_W$ of $E$ on $W$ is point-wise fixed by $\gamma$ (Lemma \ref{lem:model_change}). 
Since $W$ has a $\gamma$-fixed point, we conclude that $Y'$ also has a $\gamma$-fixed point, which proves that the $G$-action on $Y'$ is not free. 

We prove $(1) \Rightarrow (3)$.
Suppose that (3) does not hold. 
Then, there exists a $G$-equivariant proper birational morphism $Y \to X$ from a smooth variety $Y$ such that 
some closed point $y \in Y$ is fixed by some $\gamma \in G \setminus \{ 1_G \}$. 
Let $Y' \to Y$ be the weighted blow-up corresponding to the action of $\gamma$ as in Lemma \ref{lem:wb}. 
Then, the $\gamma$-action on $Y$ extends to $Y'$ and the exceptional divisor $E$ is point-wise fixed by $\gamma$. 
Note here that the entire $G$-action on $Y$ does not necessarily lift to $Y'$. 
However, by Lemma \ref{lem:equiv_resol}, we can take a $G$-equivariant proper birational morphism $W \to X$ with a prime divisor $E_W$ on $W$ such that $E_W$ is the strict transform of $E$. 
By Lemma \ref{lem:model_change}, $E_W$ is also point-wise fixed by $\gamma$. 
Therefore, the $G$-action on $W$ is not free in codimension one, 
and we conclude that the $G$-action on $X$ is not virtually free. 
\end{proof}

\begin{thm}\label{thm:prod}
Let $G$ be a finite group, and let $X$ be a normal variety with a virtually free $G$-action. 
Then, for any normal variety $Z$ with any $G$-action, the induced $G$-action on $X \times Z$ is virtually free. 
\end{thm}
\begin{proof}
By Proposition \ref{prop:char}, there exists a $G$-equivariant proper birational morphism $f: Y \to X$ from a normal variety $Y$ such that the $G$-action on $Y$ is free.
Then the morphism $f \times \mathrm{id}_Z: Y \times Z \to X \times Z$ is also a $G$-equivariant proper birational morphism. 
Moreover, since $Y$ is normal and the $G$-action on $Y$ is free, 
the product $Y \times Z$ is a normal variety and the $G$-action on $Y \times Z$ is free.
Therefore, by Proposition \ref{prop:char} again, the $G$-action on $X \times Z$ is virtually free.
\end{proof}

\begin{thm}\label{thm:kltcase}
Suppose that a finite group $G$ acts on a klt variety $X$.
If the $G$-action on $X$ is virtually free, then this action is free. 
\end{thm}
\begin{proof}
By the usual Lefschetz principle, we may reduce to the case where $k = \mathbb{C}$. 
This reduction is necessary in order to apply the results of \cite{Tak03} and to use analytic methods in the argument below. 

Suppose that the action is not free. 
Then, there exists a closed point $x \in X$ fixed by some $\gamma \in G \setminus \{ 1_G \}$. 
Let $Y \to X$ be any $G$-equivariant proper birational morphism from a normal variety $Y$. 
By Proposition \ref{prop:char}, it is sufficient to prove that the $G$-action on $Y$ is not free. 
Take a resolution $W \to Y/ \langle \gamma \rangle$ of the singularities of $Y/ \langle \gamma \rangle$.  
Let $W'$ be the normalization of $W$ in the field $k(Y)$. Then, $\langle \gamma \rangle$ acts on $W'$ and we have $W'/\langle \gamma \rangle \simeq W$ (Lemma \ref{lem:normalization}). 
Let $f$ and $g$ be the morphisms as in the following diagram. 
\[
\xymatrix{
W' \ar[r] \ar[d] \ar@/^18pt/[rr]^g & Y \ar[r] \ar[d] & X \ar[d] \\
W \ar[r] \ar@/_18pt/[rr]_f & Y/\langle \gamma \rangle \ar[r] & X/\langle \gamma \rangle
}
\]

Since $X$ is klt at $x$, $X/\langle \gamma \rangle$ is also klt at the image $x'$ of $x$. 
Then, it follows that $f^{-1}(x')$ is simply connected, by Takayama's theorem \cite{Tak03}*{Theorem 1.1} 
(and the argument in the first paragraph on page 829 of \cite{Tak03}). 
If the $\langle \gamma \rangle$-action on $g^{-1}(x)$ is free, 
then the covering map $g^{-1}(x) \to f^{-1}(x)$ should be a homeomorphism, which contradicts $\gamma \not = 1 _G$. 
Therefore, the $\langle \gamma \rangle$-action on $g^{-1}(x)$ is not free, and hence the $G$-action on $Y$ is not free. 
\end{proof}

\begin{rmk}
We would like to thank Chen Jiang and an anonymous referee for their valuable comments on a preprint version of this paper, in which they pointed out that Theorem \ref{thm:kltcase} has a purely algebraic proof, as follows. 
Let $Z \to Y$ be a $\langle \gamma \rangle$-equivariant resolution of $Y$, and let $h$
denote the composite map $Z \to Y \to X$. 
By \cite{PS16}*{Corollary 3.7} and the argument in the proof of \cite{PS16}*{Proposition 3.10}, 
one can show that $h^{-1}(x)$ contains a rationally connected $\langle \gamma \rangle$-invariant subvariety $Z'$. By applying the holomorphic Lefschetz fixed-point formula to a $\langle \gamma \rangle$-equivariant resolution of $Z'$, it follows that $Z'$ must have a $\gamma$-fixed point. 
This proves that $Y$ also has a $\gamma$-fixed point. 
\end{rmk}

\section{Proof of Theorem \ref{thm:main}}\label{section:proof}

In this section, we give a proof of Theorem \ref{thm:main}. 
First, we prove a lemma. 

\begin{lem}\label{lem:H}
Let $\xi \in k$ be a primitive third root of unity in $k$. 
Let $H \subset \mathbb{A}_k ^3$ be the hypersurface defined by $x_1 ^3 + x_2 ^3 + x_3 ^3 = 0$. 
Let $\gamma := \operatorname{diag}(1, \xi , \xi ^2) \in \operatorname{GL}_3(k)$. 
Then, the $\langle \gamma \rangle$-action on $H$ is virtually free.
\end{lem}
\begin{proof}
One can see that the induced $\langle \gamma \rangle$-action on the blow-up $\operatorname{Bl}_0 H$ of $H$ at the origin $(0,0,0)$ is free, as shown by the following argument. 

Let $(x_1, x_2 , x_3), [X_1:X_2:X_3]$ be the coordinate of the blow-up $\operatorname{Bl}_0\mathbb{A}_k ^3$. 
On the open subset $\operatorname{Bl}_0\mathbb{A}_k ^3 \cap \{ X_1 \not = 0 \} \subset \operatorname{Bl}_0\mathbb{A}_k ^3$, 
the $\gamma$-action can be described by 
$\frac{X_2}{X_1} \mapsto \xi \frac{X_2}{X_1}$ and $\frac{X_3}{X_1} \mapsto \xi ^2 \frac{X_3}{X_1}$. 
Therefore, the set of $\gamma$-fixed points of $\operatorname{Bl}_0\mathbb{A}_k ^3 \cap \{ X_1 \not = 0 \}$ 
is the curve defined by $X_2  = X_3 = x_2 = x_3 = 0$. 
On the open subset $\operatorname{Bl}_0\mathbb{A}_k ^3 \cap \{ X_1 \not = 0 \} \subset \operatorname{Bl}_0\mathbb{A}_k ^3$, 
$\operatorname{Bl}_0 H$ is defined by $1 + \left( \frac{X_2}{X_1} \right)^3 + \left( \frac{X_3}{X_1} \right)^3 = 0$. 
Therefore $\operatorname{Bl}_0 H \cap \{ X_1 \not = 0 \}$ has no $\gamma$-fixed point since $\frac{X_2}{X_1}$ or $\frac{X_3}{X_1}$ is non-zero. 
Applying the same argument to the other two open subsets $\operatorname{Bl}_0 H \cap \{ X_2 \not = 0 \}$ and $\operatorname{Bl}_0 H \cap \{ X_3 \not = 0 \}$,
we conclude that $\operatorname{Bl}_0 H$ has no $\gamma$-fixed point. 

Therefore, by Proposition \ref{prop:char}, the $\langle \gamma \rangle$-action on $H$ is virtually free. 
\end{proof}

\begin{proof}[Proof of Theorem \ref{thm:main}]
First, we prove (1). 
By the age formula (Proposition \ref{prop:age}), we have $\operatorname{mld}_{x_0}(A_n) = 2 + \frac{n}{3}$. 
Since $B_n$ is a Cartier divisor passing through $x_0$, we have 
\[
\operatorname{mld}_{x_0}(A_n, B_n) \le \operatorname{mld}_{x_0}(A_n) - 1 = 1 + \frac{n}{3}. 
\]

Next, we prove (2). 
By Lemma \ref{lem:H} and Theorem \ref{thm:prod}, the $G$-action on $\overline{B}_n$ is virtually free. 
Therefore, by Proposition \ref{prop:eq}, we have $\operatorname{mld}_{x_0}(B_n) = \operatorname{mld}_{\overline{x}_0} \bigl( \overline{B}_n \bigr)$. 
Hence, we have
\[
\operatorname{mld}_{x_0}(B_n) 
= \operatorname{mld}_{\overline{x}_0} \bigl( \overline{B}_n \bigr) 
= \operatorname{mld}_{\overline{x}_0} \bigl( \overline{B}_0 \times \mathbb{A}^n _k \bigr)
= \operatorname{mld}_{\overline{x}_0} \bigl( \overline{B}_0 \bigr) + n 
= n. 
\]
Here, the third equality follows from \cite{Amb99}*{Remark 2.7}. 
The fourth equality can be shown as follows, for instance. 
By \cite{EMY03}*{Theorem 0.1}, 
we have $\operatorname{mld}_{\overline{x}_0} \bigl( \overline{B}_0 \bigr) = 
\operatorname{mld}_{\overline{x}_0} \bigl( \mathbb{A}^3 _k, \overline{B}_0 \bigr)$. 
Since the blow-up $\operatorname{Bl}_{x_0} \mathbb{A}^3 _k \to \mathbb{A}^3 _k$ gives a log resolution of $\bigl( \mathbb{A}^3 _k, \overline{B}_0 \bigr)$, we have 
$\operatorname{mld}_{\overline{x}_0} \bigl( \mathbb{A}^3 _k, \overline{B}_0 \bigr) = 
a_E \bigl( \mathbb{A}^3 _k, \overline{B}_0 \bigr) = 0$ for the exceptional divisor $E$. 

(3) and (3') follow from (1) and (2). 
\end{proof}

\begin{rmk}\label{rmk:thin}
\begin{enumerate}
\item
We fix $n$ in Example \ref{eg:main}, and set $B = B_n$ for simplicity. 
In \cite{NS22}, the $k[t]$-scheme $\overline{B}^{(\gamma)}$ is defined by 
\[
\overline{B}^{(\gamma)}
= \operatorname{Spec} \bigl( k[t][x_1, \ldots , x_{n+3}] / (x_1^3 + t x_2 ^3 + t^2 x_3 ^3) \bigr). 
\]
It is easy to see that $x_1^3 + t x_2 ^3 + t^2 x_3 ^3 = 0$ has only the trivial solution $(x_1, x_2, x_3) = (0,0,0)$ in $x_1, x_2, x_3 \in k[[t]]$. 
This shows that the arc space $\overline{B}^{(\gamma)}_{\infty}$ is a thin set of $\overline{B}^{(\gamma)}_{\infty}$ itself (see \cite{NS22}*{Remark 5.3}).  
By \cite{NS22}*{Theorem 4.8}, the equality $\operatorname{mld}_{x_0}(B_n) = \operatorname{mld}_{\overline{x}_0} \bigl( \overline{B}_n \bigr)$  also follows from this fact. 

\item
This phenomenon ``$\overline{B}^{(\gamma)}_{\infty}$ is a thin set of $\overline{B}^{(\gamma)}_{\infty}$ itself'' 
is a pathological phenomenon in the theory of arc space of quotient varieties, 
and this does not happen when starting from a klt variety $B$. 
Actually, if $B$ is klt, \cite{NS22}*{Claim 5.2} shows that $\overline{B}^{(\gamma)}_{\infty}$ is not a thin set of $\overline{B}^{(\gamma)}_{\infty}$. 
In the proof of \cite{NS22}*{Claim 5.2}, the following two deep results are used: 
\begin{itemize}
\item
the result by Hacon and Mckernan \cite{HM07}*{Corollary 1.5}, 
which states the rational chain connectedness of the fibers of the resolution of klt singularities, and 

\item
the result by Graber, Harris, and Starr \cite{GHS03}*{Theorem 1.1}, which states the existence of a section of a morphism with rational chain connected fibers. 
\end{itemize}
The klt assumption in the result (\ref{conj:PIA}.5) in Introduction is necessary to apply these results. 
\end{enumerate}
\end{rmk}

\begin{rmk}
By \cite{Wat74}*{Theorem 1}, $A_3$ is Gorenstein since we have $\gamma \in \operatorname{SL}_6(k)$ in this case. 
Therefore, the PIA conjecture (Conjecture \ref{conj:PIA}) has a counterexample even if we assume that $X$ is Gorenstein. 
\end{rmk}

\section{Examples on semi-continuity}\label{section:LSC}

We give two examples related to the semi-continuity of minimal log discrepancies. 

\begin{eg}\label{eg:family}
Let $\xi \in k$ be a primitive third root of unity. 
We define $R := k[x_1, \ldots , x_{7}]$ and 
\[
g := x_1 ^3 + x_2 ^3 + x_3 ^3 + x_6 x_7 \in R, \qquad
\gamma := \operatorname{diag} (1, \xi, \xi^2, \xi, \xi , 1, 1) \in \operatorname{GL}_{7}(k). 
\]
We define $G := \langle \gamma \rangle = \{ 1, \gamma , \gamma^2 \} \subset \operatorname{GL}_{7}(k)$ and 
$Y := \operatorname{Spec} \bigl (R^G/(g) \bigr)$. 
Let $\varphi: Y \to \mathbb{A}^1_k = \operatorname{Spec} (k[x_7])$ be the $7$th projection. 
For a closed point $t \in \mathbb{A}^1_k$, let $Y_t := \varphi^{-1}(t)$ denote its fiber. 
Let $x_0 \in Y_t$ denote the image of the origin $\overline{x}_0 = (0, 0, \ldots, 0) \in \mathbb{A}^6 _k = \operatorname{Spec}(k[x_1, \ldots, x_6])$. 
Then, we have 
\[
\operatorname{mld}_{x_0}(Y_0) = 3 , \qquad
\operatorname{mld}_{x_0}(Y_t) = \frac{8}{3} \quad (\text{for all $t \not = 0$}), 
\tag{$\diamondsuit$}
\]
as shown by the following argument. 

We have 
\begin{align*}
Y_0 
&= \operatorname{Spec} \bigl (R^G/(g, x_7) \bigr) 
\simeq \operatorname{Spec} \bigl (k[x_1,\ldots, x_6]^G/(x_1 ^3 + x_2 ^3 + x_3 ^3) \bigr) 
\simeq B_2 \times \mathbb{A}^1_k, \\
Y_t 
&= \operatorname{Spec} \bigl (R^G/(g, x_7-t) \bigr) 
\simeq \operatorname{Spec} \bigl (k[x_1,\ldots, x_6]^G/(x_1 ^3 + x_2 ^3 + x_3 ^3+tx_6) \bigr) \\
&\simeq \operatorname{Spec} \bigl( k[x_1, x_2, x_3, x_4, x_5] ^G \bigr), 
\end{align*}
where $B_2$ is the variety defined in Example \ref{eg:main}. 
Therefore, the first assertion of ($\diamondsuit$) follows from Theorem \ref{thm:main} and \cite{Amb99}*{Remark 2.7}. 
The second assertion of ($\diamondsuit$) follows from the age formula (Proposition \ref{prop:age}). 
\end{eg}

\begin{eg}\label{eg:adic}
Let $\xi \in k$ be a primitive third root of unity. 
We define $R := k[x_1, \ldots , x_6]$ and $g, g_n \in R$ by 
\[
g := x_1 ^3 + x_2 ^3 + x_3 ^3 , \quad
g_n := x_1 ^3 + x_2 ^3 + x_3 ^3 + x_6^n \quad (\text{for $n \ge 1$}). 
\]
Let 
$\gamma := \operatorname{diag} (1, \xi, \xi^2, \xi, \xi , 1) \in \operatorname{GL}_{6}(k)$ and 
$G := \langle \gamma \rangle = \{ 1, \gamma , \gamma^2 \} \subset \operatorname{GL}_{6}(k)$. 
We define 
\[
Y := \operatorname{Spec} \bigl (R^G/(g) \bigr), \quad
Y_n :=  \operatorname{Spec} \bigl (R^G/(g_n) \bigr). 
\]
Let $x_0 \in Y, Y_n$ denote the image of the origin $\overline{x}_0 \in \mathbb{A}^6 _k = \operatorname{Spec}(k[x_1, \ldots, x_6])$. 
Then, we have 
\[
\operatorname{mld}_{x_0}(Y) = 3 , \qquad
\operatorname{mld}_{x_0}(Y_n) \le \frac{8}{3} \quad (\text{for all $n \ge 1$}), 
\tag{$\heartsuit$}
\]
as shown by the following argument. 

We have $Y \simeq B_2 \times \mathbb{A}_k ^1$, where $B_2$ is the variety defined in Example \ref{eg:main}. 
Therefore, the first assertion of $(\heartsuit)$ follows from Theorem \ref{thm:main} and \cite{Amb99}*{Remark 2.7}. 
Note that $Y_n$ is klt by \cite{Kol13}*{Corollary 2.43} and \cite{Ish96}*{Corollary 1.7(iii)}. 
Therefore, the PIA conjecture 
\[
\operatorname{mld}_{x_0}(Y_n) 
= \operatorname{mld}_{x_0} \bigl( \mathbb{A}_k ^6/G, Y_n \bigr)
\]
holds by the known result (\ref{conj:PIA}.5) in Introduction (\cite{NS22}*{Theorem 1.4}). 
Therefore, we have 
\[
\operatorname{mld}_{x_0}(Y_n) 
= \operatorname{mld}_{x_0} \bigl( \mathbb{A}_k ^6/G, Y_n \bigr) 
\le \operatorname{mld}_{x_0} \bigl (\mathbb{A}_k ^6/G \bigr) -1 
= \frac{11}{3} - 1 = \frac{8}{3}
\]
by the age formula (Proposition \ref{prop:age}). 
\end{eg}

Example \ref{eg:family} gives a counterexample to the LSC conjecture for family, which was naively believed to be correct.

\begin{conj}[LSC conjecture for families]\label{conj:LSCg}
Let $\varphi: Y \to T$ be a flat morphism, $\mathfrak{a}$ an $\mathbb{R}$-ideal sheaf on $Y$, and 
$\tau: T \to Y$ a section of $\varphi$. 
For a closed point $t \in T$, let $Y_t = \varphi ^{-1} (t)$ denote its fiber, 
and $\mathfrak{a}_t = \mathfrak{a} \mathcal{O}_{Y_t}$ denote the restriction. 
For all $t \in T$, suppose that $\mathfrak{a}_t \not = 0$ and $Y_t$ is a normal $\mathbb{Q}$-Gorenstein variety. 
Then, the function
\[
|T|_{\rm cl} \to \mathbb{R}_{\ge 0} \cup \{ - \infty \}; \quad t \mapsto \operatorname{mld}_{\tau(t)} (Y_t, \mathfrak{a}_t)
\]
is lower semi-continuous, where $|T|_{\rm cl}$ denotes the set of closed points of $T$ with the Zariski topology. 
\end{conj}

\begin{rmk}
Conjecture \ref{conj:LSCg} is known to be true if $\varphi$ is a smooth morphism. 
More generally, in \cite{Nak16}*{Theorem 3.2}, the first author proves  
Conjecture \ref{conj:LSCg} when $Y_t$ has only quotient singularities for all $t \in T$. 
\end{rmk}

\noindent
Note that the LSC conjecture (in usual form) below is still open although it is a special case of Conjecture \ref{conj:LSCg}.
\begin{conj}[LSC conjecture]\label{conj:LSC}
Let $Y$ be a normal $\mathbb{Q}$-Gorenstein variety and $\mathfrak{a}$ a non-zero $\mathbb{R}$-ideal sheaf on $Y$. 
Then, the function
\[
|Y|_{\rm cl} \to \mathbb{R}_{\ge 0} \cup \{ - \infty \}; \quad y \mapsto \operatorname{mld}_{y} (Y, \mathfrak{a})
\]
is lower semi-continuous. 
\end{conj}

Example \ref{eg:adic} gives a counterexample to the following conjectural property. 
\begin{conj}\label{conj:LSCadica}
Let $X$ be a klt variety and let $x \in X$ be a closed point.  
Let $\mathfrak{m}_x$ denote the maximal ideal corresponding to $x$. 
Then, there exists $\ell \in \mathbb{Z}_{\ge 0}$ with the following property: 
For any $\mathbb{Q}$-Gorenstein normal subvarieties $Y$ and $Z$ of $X$ satisfying 
$\dim Y = \dim Z$ and $I_Y + \mathfrak{m}_x ^{\ell} = I_Z + \mathfrak{m}_x ^{\ell}$, we have 
\[
\operatorname{mld}_x(Y) = \operatorname{mld}_x(Z). 
\]
Here, $I_Y$ and $I_Z$ denote the defining ideal sheaves of $Y$ and $Z$ in $X$. 
\end{conj}

\noindent
Conjecture \ref{conj:LSCadica} relates to the ideal-adic semi-continuity conjecture below (\cite{MN18}*{Conjecture 7.3}). 
\begin{conj}[Ideal-adic semi-continuity conjecture]\label{conj:LSCadic}
Let $X$ be a klt variety and let $x \in X$ be a closed point.  
Let $\mathfrak{m}_x$ denote the maximal ideal corresponding to $x$. 
Then, for given $r \in \mathbb{R}_{> 0}$, there exists $\ell \in \mathbb{Z}_{\ge 0}$ with the following property: 
For any ideal sheaves $\mathfrak{a}$ and $\mathfrak{b}$ on $X$ satisfying $\mathfrak{a} + \mathfrak{m}_x ^{\ell} = \mathfrak{b} + \mathfrak{m}_x ^{\ell}$, we have 
\[
\operatorname{mld}_x(X, \mathfrak{a}^r) = \operatorname{mld}_x(X, \mathfrak{b}^r). 
\]
\end{conj}

\noindent
Note that Conjecture \ref{conj:LSCadic} is known to be true if the ACC conjecture is true (see \cite{Kaw21}*{Theorem 4.6}).

\section{Questions}\label{section:Q}

It seems to be well-known to experts that the PIA conjecture (Conjecture \ref{conj:PIA}) is true when $\dim X = 3$.
We note a proof here since we could not find any reference.

\begin{prop}\label{prop:dim3}
The PIA conjecture (Conjecture \ref{conj:PIA}) is true when $\dim X =3$. 
\end{prop}
\begin{proof}
By adjunction (see \cite{92}*{17.2.1}), the inequality 
\[
\operatorname{mld}_x \bigl( X, \mathfrak{a} \mathcal{O}_X(-D) \bigr) \le \operatorname{mld}_x (D, \mathfrak{a} \mathcal{O}_D)
\tag{$\clubsuit$}
\]
holds in any case. 
Therefore, we may assume that $\operatorname{mld}_x (D, \mathfrak{a} \mathcal{O}_D) \ge 0$. 

Suppose $\operatorname{mld}_x (D, \mathfrak{a} \mathcal{O}_D) = 0$. 
Then $(D, \mathfrak{a} \mathcal{O}_D)$ is log canonical at $x$. 
Therefore, by inversion of adjunction on log canonicity (\cite{Kaw07}), 
$\bigl( X, \mathfrak{a} \mathcal{O}_X(-D) \bigr)$ is also log canonical at $x$. 
Hence, we have 
\[
\operatorname{mld}_x \bigl( X, \mathfrak{a} \mathcal{O}_X(-D) \bigr) \ge 0 = \operatorname{mld}_x (D, \mathfrak{a} \mathcal{O}_D), 
\]
which shows the opposite inequality of ($\clubsuit$). 

Suppose that $\operatorname{mld}_x (D, \mathfrak{a} \mathcal{O}_D)>0$. 
Since $\dim D = 2$ and $\operatorname{mld}_x (D) > 0$, $D$ is klt at $x$.
By inversion of adjunction, $(X,D)$ is purely log terminal in a neighborhood of $x$. 
In particular, $X$ is klt at $x$. 
When $\operatorname{mld}_x \bigl( X, \mathfrak{a} \mathcal{O}_X(-D) \bigr) \le 1$, 
the assertion follows from \cite{Kaw21}*{Theorem 7.5}. 
Suppose $\operatorname{mld}_x \bigl( X, \mathfrak{a} \mathcal{O}_X(-D) \bigr) > 1$. 
Since we have 
\[
\operatorname{mld}_x (X) \ge \operatorname{mld}_x \bigl( X, \mathcal{O}_X (-D) \bigr) + 1 > 2 
\]
in this case, $X$ is smooth at $x$ by \cite{Amb99}*{Lemma 3.2}.
Then the assertion follows from the PIA conjecture for smooth varieties (see (\ref{conj:PIA}.2) in Introduction).
We complete the proof. 
\end{proof}

Since Theorem \ref{thm:main} gives a counterexample to the PIA conjecture for $\dim X \ge 5$, 
the following question naturally arises: 

\begin{quest}\label{quest:dim4}
Is Conjecture \ref{conj:PIA} true if $\dim  X =4$? 
\end{quest}

\begin{rmk}
We would like to thank Chen Jiang and an anonymous referee for their valuable comments on a preprint version of this paper, in which they pointed out that Conjecture \ref{conj:PIA} holds for four-dimensional varieties with quotient singularities as follows. 
By the same argument as in Proposition \ref{prop:dim3}, we may assume that 
\[
\operatorname{mld}_x \bigl( X, \mathfrak{a} \mathcal{O}_X(-D) \bigr) \le \operatorname{mld}_x (D, \mathfrak{a} \mathcal{O}_D), \qquad 
\operatorname{mld}_x (D, \mathfrak{a} \mathcal{O}_D) >0.
\]
If $D$ has klt singularities, the conjecture holds by \cite{NS22}*{Theorem 1.4} (see (\ref{conj:PIA}.5) in Introduction). 
Thus, we may assume that $D$ has non-klt singularities. 
Since $\dim D = 3$ and $\operatorname{mld}_x (D) \ge \operatorname{mld}_x (D, \mathfrak{a} \mathcal{O}_D)>0$, this implies the existence of a one-dimensional log canonical center $C$ of $D$ passing through $x$. 
Then by \cite{Amb99}*{Theorem 0.1}, we have $\operatorname{mld}_x (D) \le 1$. 
Therefore, we have 
\[
\operatorname{mld}_x \bigl( X, \mathfrak{a} \mathcal{O}_X(-D) \bigr) \le \operatorname{mld}_x (D, \mathfrak{a} \mathcal{O}_D) \le \operatorname{mld}_x (D) \le 1. 
\]
Since $X$ has quotient singularities and is therefore klt, 
the conjecture follows from \cite{Kaw21}*{Theorem 7.5}. 
\end{rmk}

\noindent
Since $D = B_n$ is not klt in Example \ref{eg:main}, the following question is also reasonable. 

\begin{quest}\label{quest:PIAm}
Is Conjecture \ref{conj:PIA} true if $D$ is klt at $x$? 
\end{quest}

\appendix

\section{The exact value of $\operatorname{mld}_{x_0}(A_n, B_n)$}\label{section:exact}
In Theorem \ref{thm:main}, we have only proved the inequality $\operatorname{mld}_{x_0}(A_n, B_n) \le 1 + \frac{n}{3}$. 
However, it actually follows that
\[
\operatorname{mld}_{x_0}(A_n, B_n)  = \min \left \{ n, 1 + \frac{n}{3} \right \}. \tag{$\diamondsuit$}
\]
In this section, we give a brief proof of this equality. 

Let $C \subset \overline{B}_n$ be the subvariety defined by the ideal $(x_1, x_2, x_3) \subset R$. 
Let $\overline{V}_n \to \overline{A}_n$ and $\overline{W}_n \to \overline{B}_n$ be the blow-ups of $\overline{A}_n$ and $\overline{B}_n$ along $C$, respectively. 
Let $(x_1, \ldots , x_{n+3}), [X_1:X_2:X_3]$ be the coordinate of the blow-up $\overline{V}_n = \operatorname{Bl}_C \overline{A}_n$. 
The action of $\gamma$ on $\overline{A}_n$ extends to $\overline{V}_n$ by 
\[
X_1 \mapsto X_1, \qquad X_2 \mapsto \xi X_2, \qquad X_3 \mapsto \xi ^2 X_3.
\]
Note that $G$ also acts on $\overline{W}_n$. 

Let $V_n :=  \overline{V}_n / G$ and $W_n :=  \overline{W}_n / G$ denote the quotient varieties. 
We define $f_i$ and $\overline{f}_i$ for $i = 1,2,3,4$ as in the diagram below. 
\vspace{3mm}
\[
\xymatrix{
&&F \ar@{}[d]|{\text{\large $\cap$}} &\\
&& V_n = \overline{V}_n /G \ar[dr]^-{\overline{f}_4} & \\
\overline{F} \ar@{}[r]|{\text{\large $\subset$ \hspace{7mm}}} & \overline{V}_n = \operatorname{Bl}_C \overline{A}_n \ar[r]_-{\overline{f}_3} \ar@{}[d]|{\text{\large $\cup$}}  \ar[ur]^-{\overline{f}_2} & \overline{A}_n \ar[r]_-{\overline{f}_1} \ar@{}[d]|{\text{\large $\cup$}} & A_n \ar@{}[d]|{\text{\large $\cup$}} \\
\overline{E} \ar@{}[r]|{\text{\large $\subset$ \hspace{7mm}}}  & \overline{W}_n = \operatorname{Bl}_C \overline{B}_n \ar[r]^-{f_3} \ar[dr]_-{f_2} & \overline{B}_n \ar[r]^-{f_1} & B_n \\
&& W_n = \overline{W}_n /G \ar[ur]_-{f_4} &\\
&&E \ar@{}[u]|{\text{\large $\cup$}}&
}
\]
\vspace{3mm}

One can see that $\overline{W}_n$ is smooth, and the $G$-action on $\overline{W}_n$ is free (cf.\ Lemma \ref{lem:H}). 
Therefore, we conclude that $W_n$ is also smooth. 

Let $\overline{F}$ and $\overline{E}$ be the exceptional divisors of $\overline{f}_3$ and $f_3$, respectively. 
Let $F$ and $E$ be the images of $\overline{F}$ and $\overline{E}$, respectively. 
Then, one can see that 
\begin{itemize}
\item $\overline{E} = \overline{F} \cap \overline{W}_n$ and $E = F \cap W_n$,
\item $\overline{E}$ and $E$ are smooth, and
\item $\overline{f}^*_4 \bigl( K_{A_n} + B_n \bigr) = K_{V_n} + W_n + F$. 
\end{itemize}

Let $F_0 := \overline{f}_4 ^{-1} (x_0)$. 
Then, we have 
\[
\operatorname{mld}_{x_0}(A_n, B_n)  = \operatorname{mld}_{F_0}(V_n, W_n + F). \tag{i}
\]
Note that $F_0 \subset F$ and $\dim F_0 = 2$. 
Let $y \in F_0$ be a closed point. 

Note that $V_n$ has only quotient singularities, and $\overline{F}$ is defined by a semi-invariant equation in $\overline{V}_n$. 
Therefore, we may apply the PIA conjecture for $V_n$ and $F$, which is proved in \cite{NS3}*{Corollary 7.2}. 
Hence, we have 
\[
\operatorname{mld}_{y}(V_n, W_n + F) = \operatorname{mld}_{y}(F, W_n \cap F) = \operatorname{mld}_{y}(F, E). \tag{ii}
\]

Note that $F$ has only quotient singularities, and $\overline{E}$ is defined by an invariant equation in $\overline{F}$. 
Therefore, we may apply the PIA conjecture for $F$ and $E$, which is proved in \cite{NS22}*{Theorem 1.4} (see (\ref{conj:PIA}.5) in Introduction). 
Hence, we have 
\[
\operatorname{mld}_{y}(F, E) = 
\begin{cases}
\operatorname{mld}_{y}(E) & \text{if $y \in E$}, \\
\operatorname{mld}_{y}(F) & \text{if $y \not \in E$}. 
\end{cases} \tag{iii}
\]

We define points $\overline{p}_1, \overline{p}_2, \overline{p}_3 \in \overline{V}_n$ by 
\begin{align*}
\overline{p}_1 :=  \bigl( (0,0,0,0, \ldots, 0),[1 \colon 0 \colon 0] \bigr), \\
\overline{p}_2 :=  \bigl( (0,0,0,0, \ldots, 0),[0 \colon 1 \colon 0] \bigr), \\
\overline{p}_3 :=  \bigl( (0,0,0,0, \ldots, 0),[0 \colon 0 \colon 1] \bigr). 
\end{align*}
For each $i$, let $p_i \in F_0$ denote the image of $\overline{p}_i$. 
Then, one can see that 
\begin{itemize}
\item 
$p_1$, $p_2$ and $p_3$ are the only singular points of $F = \overline{F} /G$, and

\item 
$p_1, p_2, p_3 \in F$ has a quotient singularity of type $\frac{1}{3} \bigl( 1,2, \overbrace{1, \cdots , 1}^{n} \bigr)$. 
\end{itemize}
Therefore,  the age formula (Proposition \ref{prop:age}), we have 
\[
\operatorname{mld}_{y}(F) = 
\begin{cases}
n+2 & \text{if $y \in F \setminus \{ p_1, p_2, p_3 \}$}, \\
\frac{n}{3} + 1 & \text{if $y = p_i$ for some $i = 1,2,3$}. 
\end{cases}
\]
Furthermore, since $E$ is smooth, we have $\operatorname{mld}_{y}(E) = \dim E = n+1$ for all $y \in E$. 

Therefore, by (iii), we conclude that
\[
\operatorname{mld}_{y}(F, E) = 
\begin{cases}
n + 1 & \text{if $y \in F_0 \cap E$}, \\
\frac{n}{3} + 1 & \text{if $y = p_i$ for some $i = 1,2,3$}, \\
n+2 & \text{if $y \in F_0 \setminus (E \cup \{ p_1, p_2, p_3 \})$}. \tag{iv}
\end{cases}
\]
By (ii), (iv) and \cite{Amb99}*{Proposition 2.1}, we conclude that
\begin{align*}
\operatorname{mld}_{F_0}(V_n, W_n + F) 
&= \min \left \{ n + 1 - \dim (F_0 \cap E), \ 1 + \frac{n}{3},\ n+2 - \dim F_0 \right \} \\
&= \min \left \{ n, 1 + \frac{n}{3} \right \}. 
\end{align*}
Therefore, the assertion $(\diamondsuit)$ follows from (i).



\begin{bibdiv}
\begin{biblist*}



\bib{Amb99}{article}{
   author={Ambro, Florin},
   title={On minimal log discrepancies},
   journal={Math. Res. Lett.},
   volume={6},
   date={1999},
   number={5-6},
   pages={573--580},
}


\bib{BCHM}{article}{
   author={Birkar, Caucher},
   author={Cascini,Paolo},
   author={Hacon, Christopher D.},
   author={McKernan, James},
   title={Existence of minimal models for varieties of log general type},
   journal={J. Amer. Math. Soc.},
   volume={23},
   date={2010},
   number={2},
   pages={405--468},
}


\bib{EM04}{article}{
   author={Ein, Lawrence},
   author={Musta{\c{t}}{\v{a}}, Mircea},
   title={Inversion of adjunction for local complete intersection varieties},
   journal={Amer. J. Math.},
   volume={126},
   date={2004},
   number={6},
   pages={1355--1365},
}



\bib{EMY03}{article}{
   author={Ein, Lawrence},
   author={Musta{\c{t}}{\u{a}}, Mircea},
   author={Yasuda, Takehiko},
   title={Jet schemes, log discrepancies and inversion of adjunction},
   journal={Invent. Math.},
   volume={153},
   date={2003},
   number={3},
   pages={519--535},
}



\bib{GHS03}{article}{
   author={Graber, Tom},
   author={Harris, Joe},
   author={Starr, Jason},
   title={Families of rationally connected varieties},
   journal={J. Amer. Math. Soc.},
   volume={16},
   date={2003},
   number={1},
   pages={57--67},
}



\bib{HM07}{article}{
   author={Hacon, Christopher D.},
   author={Mckernan, James},
   title={On Shokurov's rational connectedness conjecture},
   journal={Duke Math. J.},
   volume={138},
   date={2007},
   number={1},
   pages={119--136},
}

\bib{Har77}{book}{
   author={Hartshorne, Robin},
   title={Algebraic geometry},
   note={Graduate Texts in Mathematics, No. 52},
   publisher={Springer-Verlag, New York-Heidelberg},
   date={1977},
}

\bib{Ish96}{article}{
   author={Ishii, Shihoko},
   title={The canonical modifications by weighted blow-ups},
   journal={J. Algebraic Geom.},
   volume={5},
   date={1996},
   number={4},
   pages={783--799},
}

\bib{IshiiBook}{book}{
   author={Ishii, Shihoko},
   title={Introduction to singularities},
   edition={2},
   publisher={Springer, Tokyo},
   date={2018},
   pages={x+236},
}


\bib{Kaw07}{article}{
   author={Kawakita, Masayuki},
   title={Inversion of adjunction on log canonicity},
   journal={Invent. Math.},
   volume={167},
   date={2007},
   number={1},
   pages={129--133},
}

\bib{Kaw21}{article}{
   author={Kawakita, Masayuki},
   title={On equivalent conjectures for minimal log discrepancies on smooth
   threefolds},
   journal={J. Algebraic Geom.},
   volume={30},
   date={2021},
   number={1},
   pages={97--149},
}

\bib{Kaw24}{book}{
   author={Kawakita, Masayuki},
   title={Complex algebraic threefolds},
   series={Cambridge Studies in Advanced Mathematics},
   volume={209},
   publisher={Cambridge University Press, Cambridge},
   date={2024},
}

\bib{Kaw99}{article}{
   author={Kawamata, Yujiro},
   title={Deformations of canonical singularities},
   journal={J. Amer. Math. Soc.},
   volume={12},
   date={1999},
   number={1},
   pages={85--92},
   issn={0894-0347},
}


\bib{92}{collection}{
   author={Koll{\'a}r (ed.), J{\'a}nos},
   title={Flips and abundance for algebraic threefolds},
   note={Papers from the Second Summer Seminar on Algebraic Geometry held at
   the University of Utah, Salt Lake City, Utah, August 1991;
   Ast\'{e}risque No. 211 (1992)},
   publisher={Soci\'{e}t\'{e} Math\'{e}matique de France, Paris},
   date={1992}, 
   label={Kol92}
}

\bib{Kol13}{book}{
   author={Koll{\'a}r, J{\'a}nos},
   title={Singularities of the minimal model program},
   series={Cambridge Tracts in Mathematics},
   volume={200},
   note={With a collaboration of S\'andor Kov\'acs},
   publisher={Cambridge University Press, Cambridge},
   date={2013},
}


\bib{KM98}{book}{
   author={Koll{\'a}r, J{\'a}nos},
   author={Mori, Shigefumi},
   title={Birational geometry of algebraic varieties},
   series={Cambridge Tracts in Mathematics},
   volume={134},
   publisher={Cambridge University Press, Cambridge},
   date={1998},
}



\bib{MN18}{article}{
   author={Musta\c{t}\u{a}, Mircea},
   author={Nakamura, Yusuke},
   title={A boundedness conjecture for minimal log discrepancies on a fixed
   germ},
   conference={
      title={Local and global methods in algebraic geometry},
   },
   book={
      series={Contemp. Math.},
      volume={712},
      publisher={Amer. Math. Soc., Providence, RI},
   },
   date={2018},
   pages={287--306},
}

\bib{Nak16}{article}{
   author={Nakamura, Yusuke},
   title={On semi-continuity problems for minimal log discrepancies},
   journal={J. Reine Angew. Math.},
   volume={711},
   date={2016},
   pages={167--187},
}

\bib{NS22}{article}{
   author={Nakamura, Yusuke},
   author={Shibata, Kohsuke},
   title={Inversion of adjunction for quotient singularities},
   journal={Algebr. Geom.},
   volume={9},
   date={2022},
   number={2},
   pages={214--251},
}

\bib{NS2}{article}{
   author={Nakamura, Yusuke},
   author={Shibata, Kohsuke},
   title={Inversion of adjunction for quotient singularities II: non-linear
   actions},
   journal={Algebr. Geom.},
   volume={12},
   date={2025},
   number={4},
   pages={443--496},
}

\bib{NSs}{article}{
   author={Nakamura, Yusuke},
   author={Shibata, Kohsuke},
   title={Shokurov's index conjecture for quotient singularities},
   conference={
      title={Higher dimensional algebraic geometry---a volume in honor of V.
      V. Shokurov},
   },
   book={
      series={London Math. Soc. Lecture Note Ser.},
      volume={489},
      publisher={Cambridge Univ. Press, Cambridge},
   },
   date={2025},
   pages={220--230},
}

\bib{NS3}{article}{
   author={Nakamura, Yusuke},
   author={Shibata, Kohsuke},
   title={Inversion of adjunction for quotient singularities III: semi-invariant case},
   eprint={arXiv:2312.05808v1}
}

\bib{Nakbook}{book}{
   author={Nakayama, Noboru},
   title={Zariski-decomposition and abundance},
   series={MSJ Memoirs},
   volume={14},
   publisher={Mathematical Society of Japan, Tokyo},
   date={2004},
   pages={xiv+277},
}

\bib{PS16}{article}{
   author={Prokhorov, Yuri},
   author={Shramov, Constantin},
   title={Jordan property for Cremona groups},
   journal={Amer. J. Math.},
   volume={138},
   date={2016},
   number={2},
   pages={403--418},
}

\bib{Sho93}{article}{
   author={Shokurov, V. V.},
   title={$3$-fold log flips, With an appendix by Yujiro Kawamata},
   journal={Russian Acad. Sci. Izv. Math.},
   volume={40},
   date={1993},
   number={1},
   pages={95--202},
}

\bib{Sho04}{article}{
   author={Shokurov, V. V.},
   title={Letters of a bi-rationalist. V. Minimal log discrepancies and
   termination of log flips},
   journal={Tr. Mat. Inst. Steklova},
   volume={246},
   date={2004},
   number={Algebr. Geom. Metody, Svyazi i Prilozh.},
   pages={328--351},
   translation={
      journal={Proc. Steklov Inst. Math.},
      date={2004},
      number={3 (246)},
      pages={315--336},
   },
}

\bib{Tak03}{article}{
   author={Takayama, Shigeharu},
   title={Local simple connectedness of resolutions of log-terminal
   singularities},
   journal={Internat. J. Math.},
   volume={14},
   date={2003},
   number={8},
   pages={825--836},
   doi={10.1142/S0129167X0300196X},
}


\bib{Wat74}{article}{
   author={Watanabe, Keiichi},
   title={Certain invariant subrings are Gorenstein I},
   journal={Osaka Math. J.},
   volume={11},
   date={1974},
   pages={1--8},
}

\end{biblist*}
\end{bibdiv}
\end{document}